\documentclass{amsart}
\usepackage{eucal}
\usepackage{amsmath,amsthm,amssymb}

\newcommand{\ith}{\int_{-1/2}^{1/2}}
\newcommand{\thf}[1]{\theta'#1}
\newcommand{\ths}[1]{\theta''#1}
\newcommand{\dt}{\,\mbox{d}t}
\newcommand{\ds}{\displaystyle}
\newcommand{\lp}{\mbox{L}}

\begin{document}

\title{Almost Everywhere Convergence Of Convolution Powers Without Finite Second Moment}
\author{Christopher M. Wedrychowicz} 
\address{Department of Mathematics And Statistics\\
         University At Albany\\ 
         1400 Washington Avenue, Albany, NY, 12222}
\email{cw4347@albany.edu}

\begin{abstract}
Bellow and Calder\'on proved that the sequence of convolution powers
$\ds \mu_n f(x)=\sum_{k\in\mathbb{Z}}\mu^n(k)f(T^k x)$ converges a.e, when $\ds\mu$ is a strictly aperiodic probability measure on $\ds\mathbb{Z}$ such
that the expectation is zero, $\ds E(\mu)=0$, and the second moment is finite, $\ds m_2(\mu)<\infty$. In this paper we extend this result
to cases where $\ds m_2(\mu)=\infty$.  
%Let $\ds (X,\mathcal{B},\lambda,T)$ be a dynamical system and $\ds \mbox{Log}_{(n)}x$
%be the $n-$times iterated logarithm. We will prove that given $p>0$, and $\ds n\in \mathcal{N}$ there exists an increasing sequence of non negative integers $\ds n_k$ and a function $\ds f\in \mbox{LLog}^p_{(n)}\lp(X)$ such that 
%$\ds A_Nf(x)=\frac{1}{N}\sum_{k=0}^{N-1}f(T^{n_k}x)$ fails to converge a.e. but $\ds A_N g(x)$ converges a.e. for all $\ds g\in \mbox{LLog}^q_{(n)}\lp(X)$ with $\ds q>p$.
%In the second half of this thesis we extend a theorem of Bellow and Calder\'on, which
%states that the sequence of convolution powers $\ds \mu_n f(x)=\sum_{k\in\mathbb{Z}}\mu^n(k)f(T^k x)$ converges
%a.e, when $\ds \mu$ is a strictly aperiodic probability measure on $\ds\mathbb{Z}$ such that the expectation $\ds E(\mu)=0$ and the second moment of $\mu$, $\ds m_2(\mu)<\infty$. 

\end{abstract}

\maketitle

\section{Almost everywhere convergence of convolution powers}

\subsection{Preliminaries}

Let $\ds\mu(k)$, $\ds k\in\mathbb{Z}$, be a probability measure. It's convolution 
$\ds\mu\ast\mu$ is defined by $\ds\mu\ast\mu(k)=\sum_{j\in\mathbb{Z}}\mu(j)\mu(k-j)$.
The $n-$fold convolution \\
$\ds\mu\ast\cdots\ast\mu(k)=\mu^n(k)$ is defined inductively by
$\ds\mu^n(k)=\mu\ast\mu^{n-1}(k)$. The Fourier transform of $\mu$ will be denoted by $\ds\theta(t)$ for $\ds t\in[-1/2,1/2)$ and it is equal to
$$\theta(t)=\sum_{k\in\mathbb{Z}}\mu(k)e^{2\pi ikt}.$$
The weights may be recovered from the inversion formula
$$\ds \mu(k)=\int_{-1/2}^{1/2}\theta(t)e^{-2\pi ikt}\,\mbox{d}t.$$
The $p^{\mbox{th}}$ moment of $\mu$ is the sum $\ds m_p(\mu)=\sum|k|^p\mu(k)$ and we say
that $\mu$ has a $p^{\mbox{th}}$ moment if the above sum is finite. It will be necessary to consider positive non-integral moments $p$.
The expectation of $\mu$ is given by $\ds E(\mu)=\sum_{k\in\mathbb{Z}}k\mu(k)$.
\theoremstyle{definition}
\newtheorem{straper}{Definition}[section]
\begin{straper}
$\ds\mu$ is called strictly aperiodic if the support of $\ds\mu$ is not contained in 
a proper coset of the integers. 
\end{straper}

\vspace{.1in} 

\noindent We have the following important theorem by Foguel(~\cite{foguel}).
\theoremstyle{plain}
\newtheorem{foguel}{Theorem}[section]
\begin{foguel}
$\ds\mu$ is strictly aperiodic $\iff$ $\ds|\theta(t)|\neq 1$ $\ds\forall t\neq 0$.
\end{foguel}

Now let $\ds T:X\rightarrow X$ be an invertible measure preserving transformation of a probability space $\ds(X,\mathcal{B},\lambda).$ For $\ds f\in\mbox{L}^{1}(\lambda)$ one may define 
$$\mu f(x)=\sum_{k\in\mathbb{Z}}\mu(k)f(T^k x).$$
\noindent One then considers the question of a.e convergence for the sequence \\
$\ds\mu^n f(x)=\sum_{k\in\mathbb{Z}}\mu^n(k)f(T^k x)$.

\theoremstyle{definition}

\newtheorem{bddang}[foguel]{Definition}
\begin{bddang}
A probability measure $\mu$ on $\ds \mathbb{Z}$ has bounded angular ratio, if
$\mu$ is strictly aperiodic and $$1\leq \sup_{t\neq 0}\frac{|\hat{\mu}(t)-1|}{1-|\hat{\mu}(t)|}<\infty$$
\end{bddang}

\newtheorem{oper}[foguel]{Definition}
\begin{oper}
We denote by $\ds \mu^n\,:\lp^1(X)\rightarrow\lp^1(X)$ the operator defined by
$\ds\mu^n f(x)=\sum_{k\in\mathbb{Z}}\mu^n(k)f(T^k x)$. We refer to the sequence $\ds \mu^n$
as the sequence of convolution powers of $\ds\mu$.
\end{oper}

\noindent The following Theorem establishes the necessity of of the bounded angular ratio
condition in the study of convolution powers.
\theoremstyle{plain}
\newtheorem{lost}[foguel]{Theorem}
\begin{lost}[\cite{losertsec}]Suppose that $\mu$ is a probability measure on $\mathbb{Z}$($\ds \mu\neq \delta_{k}$, i.e not concentrated
in a single point) and that $\hat{\mu}$ has unbounded angular ratio. Then there exists a function $\ds f\in\mbox{L}^{\infty}$
such that $\ds \mu_nf(x)$ fails to converge a.e.
\end{lost}

\noindent In order to establish some key properties of the measures under our consideration
we will need the following Theorem from \cite{young}; it provides a generalization of L'Hospital's rule concerning indeterminate forms of type $\ds 0/0$.

\newtheorem{ctsdif}[foguel]{Theorem}
\begin{ctsdif}
\label{ctsdif}
Let $f(x)$ and $F(x)$ be continuous differentiable functions on an open interval $\ds (a,x)$
where $a$ may denote $-\infty$, and the differential coefficients $f'$ and $F'$ have no common zeros or infinities in the open interval and if, as $x$ approaches the value $a$, $f(x)$ and $F(x)$ each have the unique limit zero, then as $x$ approaches the value $a$ the limits of
$$\ds\frac{f(x)}{F(x)}$$ are the limits of $$\ds \frac{f'(x)}{F'(x)}.$$

\end{ctsdif}

\newtheorem{newlemma}[foguel]{Lemma}
\begin{newlemma}
\label{newlemma}
Suppose $\mu$ is a probability measure defined on $\ds \mathbb{Z}$ such that $\ds m_1(\mu)<\infty$, $\mu$ has bounded angular ratio and $\ds \theta''(t)$ exists on a set $\ds S=(-\delta,\delta)-\{0\}$. Furthermore if $\ds \theta(t)=f(t)+ig(t)$ we have 
%$\ds f''(t)=p(t)+O(1)<0$ on $S$ where 
$\ds f''(t)<0$ on $S$ then
%$\ds p(t)$ is non-increasing in $|t|$ then
\begin{itemize}
\item[(a)] $\ds E(\mu)=0$\, , 
\item[(b)] $\ds |g'(t)|\leq c_1|f'(t)|$ and $\ds |g''(t)|\leq c_2|f''(t)|$ on $S$.
\end{itemize}
\end{newlemma}

\begin{proof}
(a) is a special case of a result proved in \cite{belrjblatt}(Proposition 1.9). To prove (b) we first note that by \cite{losert}(Lemma 1)
$$\ds \limsup_{t\rightarrow 0}\frac{|1-\hat{\mu}(t)|}{1-|\hat{\mu}(t)|}<\infty\Leftrightarrow
\limsup_{t\rightarrow 0}\left|\frac{g(t)}{1-f(t)}\right|<\infty$$
Since $f''(t)<0$ on $S$ and $\ds f'(0)=0$, $f'(t)$ has no zeros on $S$ except when $t=0$. Therefore by Theorem \ref{ctsdif} the limit points of $\ds \frac{g(t)}{1-f(t)}$ coincide with that of $\ds \frac{g'(t)}{-f'(t)}$ and $\ds\frac{g''(t)}{-f''(t)}$, which proves (b).
\end{proof}

\theoremstyle{definition}
\newtheorem{defmax}[foguel]{Definition}
\begin{defmax}
Given a sequence of operators $\ds T_n: \lp_1\rightarrow \mathcal{M}(X)$, where $\ds\mathcal{M}(X)$ denotes the set of measurable functions, the maximal operator $\ds T^{\ast}$ of $\ds\{T_n\}$
is defined by $\ds T^{\ast}f(x)=\sup_n|T_n f(x)|$. For the sequence $\ds\{\mu^n\}$ the maximal operator will be denoted by $\ds\mu^{\ast}$.
\end{defmax}

\theoremstyle{plain}

\newtheorem{ban}[foguel]{Theorem}
\begin{ban}
\label{ban}
Let $\ds (X,\mathcal{B}, m)$ be a probability space. If $\ds \{T_{n}\}$ is a sequence of bounded operators such that 
\begin{eqnarray*}
T^{\ast}f(x)=\sup_{n}|T_{n}f(x)|<\infty \quad \mbox{a.e.}
\end{eqnarray*}
for every $\ds f\in \lp_1$ then the set of functions in $\ds\lp_1$ such that $\ds T_{n}f(x)$ converges a.e. is closed.
\end{ban}

In order to show $\ds T^{\ast}f(x)<\infty$ a.e, one often establishes a weak maximal inequality for the operator $\ds T^{\ast}$ of the form

\begin{eqnarray*}
m(\{x:T^{\ast}(x)\geq \lambda \})\leq C\frac{\|f\|_1}{\lambda}
\end{eqnarray*}

where $\ds C$ is a constant independent of $f$ and $\lambda$.

\theoremstyle{plain}
%\newtheorem{propl}[foguel]{Proposition}
%\begin{propl}[~\cite{belrjblatt}]
%If $\ds\mu$ is strictly aperiodic, then
%$\ds\lim_{n\rightarrow\infty}\|\mu^n\ast\delta_1-\mu^n\|_{l^1(\mathbb{Z})}=0$. 
%\end{propl}

\newtheorem{thml}[foguel]{Theorem}
\begin{thml}[~\cite{belrjblatt}]
If $\ds\mu$ is strictly aperiodic, $\ds\mu_n f(x)$ converges a.e. for all\\
$\ds f\in S=\{(f_1\circ T-f_1)+f_2\, : f_1\in\lp_{\infty}(X),\: f_2(T x)=f_2(x)\,\mbox{a.e.}\}$, which is a dense set in $\ds \lp^p$ for all $\ds p\geq 1$. 
\end{thml}

\theoremstyle{remark}
\newtheorem{rem1}[foguel]{Remark}
\begin{rem1}
The above two Theorems imply that in order to establish a.e convergence of $\ds\mu^n f(x)$ for
a strictly aperiodic measure $\ds\mu$ we need only show that $\ds\mu^{\ast}$ satisfies a weak $\ds\lp_1$ inequality
of the form given above.
\end{rem1}

\noindent In ~\cite{cald} it was shown that in order to establish weak maximal inequalities for operators on continuous spaces it is enough to establish them for an operator on $\ds l^1(\mathbb{Z})$
that has been transfered from the continuous space. Such processes are known collectively 
as the Calder\'on Transfer Principle. In our case we make use of the following version.

\theoremstyle{plain}
\newtheorem{caldtransf}[foguel]{Theorem}
\begin{caldtransf}[Calder\'on Transfer Principle,~\cite{caldzyg}]
Consider the dynamical system \\
$\ds (\mathbb{Z},\mathbb{P},|\cdot|,T)$ where $\ds |B|=\#$of elements in
$\ds B$, $\ds\mathbb{P}=$Power set of $\ds\mathbb{Z}$ and $\ds T(x)=x+1$. For $\ds\forall\phi\in l_1(\mathbb{Z})$ we have 
$\ds\mu^n\phi(k)=(\mu^n\ast\phi)(k)=\sum_{j\in\mathbb{Z}}\mu^n(j)\phi(k-j)$ and $\ds (M\phi)(k)=\sup_n (\mu^n\phi)(k)$. Then if
$$\left|k\in\mathbb{Z}:|(M\phi)(k)|\geq t\right|\leq C\frac{\|\phi\|_{l_1}}{t}$$ 
we have 
$$\lambda\left(\left\{x\in X:|(Mf)(x)|\geq t\right\}\right)\leq C\frac{\|f\|_{1}}{t}$$
for all $\ds (X,\mathcal{B},\lambda,T)$ and for all $\ds f\in\lp_1(X)$.
\end{caldtransf}

\noindent In light of the above, the following general result, which we shall use, was obtained in \cite{cald}.

\newtheorem{belcald}[foguel]{Theorem}
\begin{belcald}
\label{belcald}
Let $\ds (\mu_n)$ be a sequence of probabilities on $\ds\mathbb{Z}$ and for 
$\ds \phi\in\lp_1(\mathbb{Z})$ define the maximal operator 
$$(M\phi)(x)=\sup_n|(\mu_n\phi)(x)|\:,\quad  x\in \mathbb{Z}$$
Assume that there is $\ds 0<\alpha\leq 1$ and $\ds C''>0$ such that for each $\ds n\geq 1$
\begin{eqnarray}
|\mu_n(x+y)-\mu_n(x)|\leq C''\frac{|y|^{\alpha}}{|x|^{1+\alpha}},\:\mbox{for}\: x,y\in\mathbb{Z},\:\mbox{and}\: 2|y|\leq |x|
\end{eqnarray}
Then the maximal operator $M$ ie weak type $\ds (1,1)$,i.e there is $\ds C>0$
such that for any $\ds \lambda>0$
$$m\left\{x\in X:(M\phi)(x)>\lambda\right\}\leq\frac{C}{\lambda}\|f\|_1\:\mbox{for}\,\mbox{all}\: \phi\in\mbox{L}^{1}(X)$$ 
\end{belcald}

%In ~\cite{} it was shown that if $\ds E(\mu)=0$ and $\ds m_2(\mu)<\infty$ then $\ds\mu^n f$ converges a.e for all $\ds f\in\mbox{L}_p$ for $p>1$,
%a result that was extended to the case $p=1$ in ~\cite{}. 
Using the above, in ~\cite{cald}, Bellow and Calder\'on established the a.e convergence 
of $\ds\mu^n f(x)$ when $\ds\mu$ is strictly aperiodic, $\ds E(\mu)=0$ and $\ds m_2(\mu)<\infty$.
In ~\cite{losert} Losert shows that this result can not be extended to measures with $\ds m_p(\mu)<\infty$ for $\ds p<2$,
however, in ~\cite{belrjblatt} it was shown that if $\mu$ is symmetric and $\ds \mu(k)\geq \mu(k+1)$ for all $\ds k>0$
then $\mu^n f(x)$ converges a.e $\ds \forall f\in\mbox{L}^1(X)$. This implies that the measure given by $\ds\mu(k)=\frac{c}{|k|\log^2|k|}$
yields a.e convergence of $\ds \mu^n f(x)$ even though $\ds m_p(\mu)=\infty$ for all $\ds p>0$. The goal of this paper will be to extend the result in ~\cite{cald}
by weakening the second moment condition.

\noindent Translating the second moment condition into a statement concerning Fourier transforms we obtain 
\begin{eqnarray*}
\mu\:\mbox{has}\:\mbox{finite}\:\mbox{second}\:\mbox{moment}\:\Leftrightarrow\theta\:\mbox{is}\:\mbox{twice}\:\mbox{continuously}\:\mbox{differentiable}
\end{eqnarray*} 
Therefore we will seek a condition weaker than a continuous second derivative. The most obvious condition
would be that $\ds\theta'(t)\in\mbox{Lip}_{1}[-1/2,1/2]$, however
the following (\cite{moricz}) shows this extension to be vacuous,

\newtheorem{liplem}[foguel]{Theorem}
\begin{liplem}[Moricz]
Let $\ds f(x)=\sum_{k\in\mathbb{Z}}c_ke^{ikx}$. If $\ds \{c_k\}\subset\mathbb{C}$ is such that
$\ds\sum_{|k|\leq n}|kc_k|=O(n^{1-\alpha})$, $\ds n=1,2,\ldots,$ for some $\ds 0<\alpha\leq 1$,
then $\ds f\in\mbox{Lip}(\alpha)$. Conversely, let $\ds c_k$ be a sequence of real numbers such that $\ds k c_k\geq 0$
for all $\ds k\in\mathbb{Z}$. If $\ds \sum_{k\in\mathbb{Z}} |c_k|$ is finite and $\ds f\in\mbox{Lip}(\alpha)$ for
some $\ds 0<\alpha\leq 1$, then $\ds\sum_{|k|\leq n}|kc_k|=O(n^{1-\alpha})$. 
\end{liplem}

Although the extension $\ds \thf(t)\in\mbox{Lip}(1)$ is vacuous, the condition $\ds\thf(t)\in\mbox{Lip}(\alpha)$
for some $\ds 0<\alpha< 1$ will not be. We will construct examples of non-symmetric measures $\ds\mu$ with $\ds m_2(\mu)=\infty$ with $\ds\{\mu^n f(x)\}$
converging a.e. In fact, given $\ds p>1$, we will give examples of non-symmetric $\mu$
with $\ds E(\mu)=0$, $\ds m_p(\mu)=\infty$ and $\ds \{\mu^n f(x)\}$ converging a.e.

%In \cite{cald} the fact that measures having finite second moment and mean $0$ also satisfy the conditions of Theorem~\ref{belcald}
%boiled down to showing that 
%$$ \sup_n\int|\theta_n''(t)||t|\,\mbox{d}t<\infty$$ 
%where $\ds\theta_n(t)=\theta^{n}(t).$
%The results that follow will depend on achieving a similar type of bound. 

We will need the following.

\newtheorem{petrovlem}[foguel]{Theorem}
\begin{petrovlem}{(\cite{petrov})}
\label{petrov}
Let $\ds\theta(t)$ be the Fourier transform of a measure $\ds\mu$ on $\ds\mathbb{Z}$ not supported at a single integer.
Then there exist positive constants $\ds\delta$ and $\ds\epsilon$ such that $\ds|\theta(t)|\leq 1-\epsilon t^2$ for $\ds|t|\leq\delta$. Therefore 
for $\ds \mu$ strictl y aperiodic there exists a $C$
such that $\ds |\theta(t)|\leq e^{-Ct^2}\,\:\forall t\in[-1/2,1/2)$. 
\end{petrovlem}

\newtheorem{lemmathree}[foguel]{Lemma}
\begin{lemmathree}{(\cite{cald})}
\label{caldlem}
There is a constant $C>0$ such that, for any $\ds x,y\in\mathbb{Z}$, $\ds 0<2|y|<|x|$,
and $\ds t\in\mathbb{R}$
$$\left|\frac{e((x+y)t)-1}{(x+y)^2}-\frac{e(xt)-1}{x^2}\right|\leq C|t|\frac{|y|}{|x|^2}\: ,$$
where $\ds e(x)=e^{2\pi ix}$.
\end{lemmathree}

\noindent This paper contains results from the author's Ph.D dissertation. He would like to thank his adviser, Dr. Karin Reinhold,
for all her help and insightful comments throughout the process. The author would also like to thank the referee for the insightful comments and suggestions. In particular that of the use of the bounded angular ratio condition which appears throughout this paper, and how it could be used in the proof of Proposition ~\ref{resultcom}.

\section{Main Results}

Throughout we suppose that all measures $\mu$ have bounded angular ratio and $\ds m_1(\mu)<\infty$. Note that a constant $c$, independent of certain quantities, may change throughout an argument. Our main results are the following.

\theoremstyle{plain}
\newtheorem{mainsec}{Theorem}[section]
\begin{mainsec}
\label{mainsec}
Suppose $\mu$ is a probability on $\ds\mathbb{Z}$ with $\ds m_1(\mu)<\infty$ and 
bounded angular ratio and for some $\ds 0<\alpha\leq 1$ $\ds \sum_{|k|\leq n}k^2\mu(k)=O(n^{1-\alpha})$.
Suppose $\ds \ths(t)$ exists in some set $\ds 0<|t|<\delta$ and $\ds \mbox{Re}(\ths(t))=p(t)+O(1)<0$, where $\ds p(t)$ is non-decreasing in this set. Then $\ds \{\mu^n f(x)\}$ converges a.e. for all $\ds f\in\lp^{1}(X)$.
\end{mainsec}

\noindent Since symmetric, strictly aperiodic measures have real-valued Fourier transform they have bounded angular ratio $1$. Hence, we have the following corollary.

\newtheorem{mainseccor}[mainsec]{Corollary}
\begin{mainseccor}
Suppose $\mu$ is a strictly aperiodic, symmetric measure on $\ds\mathbb{Z}$ and for some $\ds 0<\alpha\leq 1$ $\ds \sum_{|k|\leq n}k^2\mu(k)=O(n^{1-\alpha})$.
Suppose $\ds \ths(t)$ exists in some set $\ds 0<|t|<\delta$, and $\ds \ths(t)=p(t)+O(1)<0$, where $\ds p(t)$ is non decreasing in this set. Then $\ds \{\mu^n f(x)\}$ converges a.e. for all $\ds f\in\lp^{1}(X)$.
\end{mainseccor}

%\theoremstyle{plain}
%\newtheorem{corbjr}[mainsec]{Corollary}
%\begin{corbjr}
%\label{cor}
%If $\ds\mu$ is symmetric, strictly aperiodic and $\ds\sum_{|k|\leq n}|k|^2\mu(k)=O(n^{1-\delta})$
%for some $\ds 0<\delta\leq 1$ and $\ds\ths(t)=p(t)+O(1)(t)$ with $p(t)$ as above
%then $\ds\mu^n f(x)$ converges a.e.
%\end{corbjr}

\noindent The following gives examples of non symmetric measures $\mu$ with $\ds m_2(\mu)=\infty$ and 
$\ds\mu_n f(x)$ converging a.e. for all $\ds f\in\lp^1(X)$.

\theoremstyle{definition}

\newtheorem{ex1}[mainsec]{Example}

\begin{ex1}
Let $\ds \eta(k)=s/|k|^3$ , $\ds k\neq 0$ where $\ds s=\left(\sum 1/(|k|^3)\right)^{-1}$. Then
$\ds\hat{\eta}(t)=\sum_{k>0}\frac{s}{|k|^3}\cos(2\pi kt)$ and therefore
$\ds\hat{\eta}''(t)=-4\pi^2 s\sum_{k>0}1/k\cos(2\pi kt)$. We have $\ds\sum_{|k|\leq n}1/|k|=O(\log(n))$ and 
$\ds\sum\frac{1}{k}\cos(2\pi kt)=\log\left(\frac{1}{|2\sin(x/2)|}\right)$, (see \cite{zyg}). Hence, $\ds \eta''(t)=-4\pi^2 s\log\left(\frac{1}{|2\sin(t/2)|}\right)$
is monotone in a neighborhood of $0$. If $\ds\nu$ is a measure with $\ds E(\nu)=0$ and $\ds m_2(\nu)<\infty$ then
$\ds \mu=a_1\eta+(1-a_1)\nu$ with $\ds 0<a_1\leq 1$, will have $\ds\ths(t)=p(t)+O(1)(t)$ and by Theorem~\ref{mainsec} $\ds (\mu^nf)(x)$ converges a.e.
Note here that $\ds m_p(\mu)$ is finite if and only if $\ds p<2$.
\end{ex1}

\vspace{.01in}

\newtheorem{ex2}[mainsec]{Example}
\begin{ex2}
Let $\ds\eta(k)=s/|k|^{2+\sigma}$ for some $\ds 0<\sigma<1$, $\ds k\neq 0$ and $\ds s=\left(\sum 1/|k|^{2+\sigma}\right)^{-1}$.
Since $\ds\sum_{|k|\leq n}1/|k|^{\sigma}=O(n^{1-\sigma})$ and $\ds \hat{\eta}''(t)=$\\
$\ds\sum_{k=1}^{\infty}\frac{\cos(2\pi kt)}{k^{\alpha}}=\Gamma(1-\alpha)\sin(\frac{1}{2}\pi\alpha)t^{\alpha-1}+O(1)$, (see \cite{zyg}).
Then we can construct a measure $\mu=a_1\eta+(1-a_1)\nu$ where $\ds 0<a_1\leq 1$ and $\ds\nu$ is as in the previous example such that
$\ds \mu_n f(x)$ converges a.e. Note here that $\ds m_p(\mu)$ is finite if and only if $\ds p<1+\sigma$. 
\end{ex2}

\theoremstyle{plain}
\newtheorem{result1}[mainsec]{Proposition}
\begin{result1}
Suppose that the measure $\ds\mu$ satisfies the condition $\ds \sum_{|k|\leq n}k^2\mu(k)=O(n^{1-\delta})$, so that the Fourier 
transform $\ds\theta(t)$ has $\ds\thf(t)\in\mbox{Lip}_{\delta}$, for some $\ds 0<\delta\leq 1$, then
$$|\mu^n(x)|\leq c\left\{\frac{\sqrt{n}}{|x|^{1+\delta}}+\frac{n^2}{|x|^2}\right\}\: .$$

\end{result1}

\begin{proof}
\begin{eqnarray*}
|\mu^n(x)| & = & \left|\ith\theta^{n}(t)e^{-2\pi ixt}\,\mbox{d}t\right|\: ,\quad\mbox{by integration by parts}\\
           & = & \frac{1}{2\pi}\left|\frac{n}{x}\ith\theta^{n-1}(t)\thf(t)e^{-2\pi ixt}\,\mbox{d}t\right|\\
           & \leq & c\frac{n}{|x|}\left|\ith(\theta^{n-1}(t+h)\thf(t+h)-\theta^{n-1}(t)\thf(t))e^{-2\pi ixt}\,\mbox{d}t\right|,\quad \mbox{where}\: h=\frac{1}{2x}\\
           & \leq & c\frac{n}{|x|}\ith |\theta^{n-1}(t+h)||\thf(t+h)-\thf(t)|+|\thf(t)||\theta^{n-1}(t+h)-\theta^{n-1}(t)|\,\mbox{d}t \\
           & = & c\frac{n}{|x|} I_1 +c\frac{n}{|x|} I_2\: .
\end{eqnarray*}

Now we examine $\ds I_1$ and $\ds I_2$ separately. By the Lipschitz property of $\ds\thf(t)$
\begin{eqnarray*}
I_1 & = & \frac{1}{x^{\delta}}\ith |\theta^{n-1}(t+h)|x^{\delta}|\thf(t+h)-\thf(t)|\,\mbox{d}t\\
& \leq & \frac{c}{x^{\delta}}\ith |\theta^{n-1}(t+h)|\,\mbox{d}t\\
& \leq &  \frac{c}{x^{\delta}}\frac{1}{\sqrt{n}}\: ,\quad\mbox{by}~\ref{petrov}\: .
\end{eqnarray*}

Therefore,
$$ \frac{cn}{|x|}I_1\leq \frac{c\sqrt{n}}{x^{1+\delta}}\: ,$$

\begin{eqnarray*}
I_2 & = & \frac{1}{|x|}\ith |\thf(t)||x||\theta^{n-1}(t+h)-\theta^{n-1}(t)|\,\mbox{d}t\\
& \leq & c\frac{n-1}{|x|}\ith |\thf(t)||\thf(c(t))||\theta^{n-2}(c(t))|\,\mbox{d}t\leq \frac{cn}{|x|}
\end{eqnarray*}
\noindent for some $\ds c(t)$ between $t$ and $t+h$ by the mean value theorem.

\noindent Therefore 
$$\frac{cn}{|x|}I_2\leq \frac{cn^{2}}{|x|^2}.$$
Hence
$$|\mu_n(x)|\leq\frac{c_1\sqrt{n}}{|x|^{1+\delta}}+\frac{c_2 n^{2}}{|x|^2}\: .$$

\end{proof}

\newtheorem{cor}[mainsec]{Corollary}
\begin{cor}
\label{corm}
Let $\ds\sigma=\min\left\{\frac{15\delta}{16},\frac{3}{4}\right\}$ then $\ds\sigma>0$ and 
if $\ds n\leq |x|^{\frac{\delta}{8}}$, 
$$ |\mu_n(x)|\leq \frac{C}{|x|^{1+\sigma}}\leq \frac{C|y|^{\sigma}}{|x|^{1+\sigma}},$$ $\ds\forall y\in\mathbb{Z}$, $\ds y\neq 0$.
\end{cor}

\newtheorem{resultcom}[mainsec]{Proposition}
\begin{resultcom}
\label{resultcom}
Suppose there is a set $(-\delta,\delta)-\{0\}$ on which $\ds\theta(t)$ is twice differentiable
%at all points except $0$, 
and assume that if $\ds\theta(t)=f(t)+i g(t)$, 
we have $\ds f''(t)=p(t)+O(1)<0$ where $p(t)$ is non-decreasing on $\ds (-\delta,\delta)-\{0\}$
and $\ds p(t)\rightarrow -\infty$ as $\ds t\rightarrow 0$.
%we have
%\begin{enumerate}
%\item $\ds f''(t)\rightarrow -\infty$ as $\ds t\rightarrow 0$ and $\ds g''(t)\rightarrow 0$, $\ds g''(0)=0$
%\item $\ds\left|\frac{f''(t/2)}{f''(t)}\right|\geq c\: .$
%\end{enumerate}
Then, for $\ds t$ in $\ds (-\delta,\delta)$, there exists a positive constant $c$ such that
$$|\theta(t)|\leq 1-c\phi(t)t^2$$ 
where $\ds\phi(t)=\left|\frac{f'(t)}{t}\right|$.
\end{resultcom}

\begin{proof}
By bounded angular ratio 
%and the mean value theorem we have 
\begin{eqnarray*}
|\theta(t)| & \leq & 1-c|1-\theta(t)| \\
            & \leq & 1-c|1-f(t)|. \\
\end{eqnarray*} 
%by the assumption on $\ds f''(t)$ we have 
Since $\ds f''(t)<0$, $\ds f'(0)=0$ and $\ds p(t)\rightarrow -\infty$
implies $\ds \left|\frac{f'(t/2)}{f'(t)}\right|$ is bounded below, we have that the above quantity
%\begin{eqnarray*}
 %& \leq & 1-\frac{c|t|}{2}\frac{|f'(t)|}{2}
 %& \leq & 1-\frac{c'}{2}|t||f'(t)| \\
 %& \leq & 1-c''t^2\left|\frac{f'(t)}{t}\right|\, .
%\end{eqnarray*}

\begin{eqnarray*}
 & = & 1-c\left|\int_{0}^{t} f'(s)\,\mbox{d}s\right| \\
 & \leq & 1-c\left|\int_{t/2}^{t}f'(s)\,\mbox{d}s\right| \\
 & \leq & 1-c\frac{|t|}{2}\left|f'(t/2)\right| \\
 & \leq & 1-c't^2\left|\frac{f'(t)}{t}\right|\, .
\end{eqnarray*}
\end{proof}

\theoremstyle{plain}
\newtheorem{newlem}[mainsec]{Lemma}
\begin{newlem}
\label{lemone}
Suppose that $\ds \mbox{Re}(\ths(t))=f''(t)=p(t)+O(1)<0$ 
%and $\ds |Im(\theta''(t))|\leq c|Re\theta''(t)|$,
on a set $\ds S=(-\delta,\delta)-\{0\}$ where $p(t)$ is non-increasing as $\ds |t|\rightarrow 0$, and $\ds p(t)\rightarrow -\infty$ as $\ds t\rightarrow 0$. Then the function defined by $\ds\phi(t)=\left|\frac{f'(t)}{t}\right|$ 
%if $\ds \theta''(t)$ is unbounded on $S$ and 
%$\ds \phi(t)=\sup_{t\in S}\{|\theta''(t)|,\left|\frac{\theta'(t)}{t}\right|\}$ if $\theta''(t)$ is bounded on $S$ (This is true when $m_2(\mu)<\infty$) 
satisfies the following properties for all $\ds t$ in some set $\ds [-\delta',\delta']-\{0\}$,
\begin{enumerate}
\item $\ds\phi(t)=\phi(-t)$ 
\item There exist constants $\ds 2>c_1>1$, $\ds c_2>0$, $\ds c_3>0$ such that 
$\ds c_1\phi(t)\geq |f''(t)|$, $\ds c_2\phi(t)\geq |\theta''(t)|$, $\ds c_3\phi(t)\geq \left|\frac{\theta'(t)}{t}\right|$ 
%$\ds c\phi(t)\geq |\ths(t)|$ for $\ds 1<c<2$ ; $\ds c\phi(t)\geq \left|\frac{\thf(t)}{t}\right|$ for $\ds c>0$
\item $\ds |t\phi(t)|\rightarrow 0$ as $\ds t\rightarrow 0$ and
$\ds |t\phi'(t)|\leq \phi(t)$
\end{enumerate}
\end{newlem}

\begin{proof}
%If $\theta''(t)$ is bounded on $S$ then the assertion follows from the definition of $\phi(t)$. Now suppose that $\theta''(t)$ is
%unbounded on $S$. Then
We have $\ds f(t)=\sum_{k\geq 0}c_k\cos(2\pi kt)$ with $\ds c_k\geq 0$ and therefore the first assertion is trivial.
Since $\ds f'(0)=0$, $\ds|t\phi(t)|\rightarrow 0$.
%Suppose the second moment exists. Then $\ds\lim_{t\rightarrow 0}\frac{f'(t)}{t}=\lim_{t\rightarrow 0}f''(t)\neq 0$ and
%$\ds \lim_{t\rightarrow 0}\frac{f'(t)}{tf''(t)}=1$. So $\ds |f''(t)|\leq c\left|\frac{f'(t)}{t}\right|$ for some  $c$
%as close to $1$ as we like. Note that this may require taking a smaller $\ds\delta$. Now suppose that 
%If $\ds \theta''(t)$ is unbounded on $S$, then $\ds f''(t)=p(t)+O(1)$
%where $\ds p(t)\rightarrow -\infty$ monotonically as $\ds |t|\rightarrow 0$. 
We have, as $\ds f'(0)=0$ for $\ds c(t)$
between $0$ and $t$, 
\begin{eqnarray*}
\phi(t) & = & \left|\frac{f'(t)}{t}\right|=|f''(c(t))|=|p(c(t))+O(1)|\\
& \geq & |p(t)+O(1)|
\end{eqnarray*}  
Hence for some $\ds 2>c_1>1$, since $\ds \phi(t)\rightarrow\infty$ as $\ds t\rightarrow 0$,
\begin{eqnarray*}
c_1\phi(t) & \geq & |p(t)|+(c_1-1)\phi(t)-|O(1)|\\
& \geq & |p(t)|+|O(1)|\geq |p(t)+O(1)|\\
& = & |f''(t)|.
\end{eqnarray*}
Since by~\ref{newlemma} (b) $\ds |g''(t)|\leq c|f''(t)|$, we have $\ds |g''(t)|\leq c'\phi(t)$. So $\ds |\ths(t)|\leq |g''(t)|+|f''(t)|\leq c_2\phi(t)$. Observe that,
by ~\ref{newlemma} (b)
\begin{eqnarray*}
\left|\frac{\thf(t)}{t}\right| & \leq & \left|\frac{g'(t)}{t}\right|+\left|\frac{f'(t)}{t}\right|\\
                               & \leq & c\left|\frac{f'(t)}{t}\right|+\left|\frac{f'(t)}{t}\right|\\
           %                    & = & g''(s) +\phi(t) \\
            %                   & \leq & c|f''(s)|+\phi(t)\\
             %                  & \leq & c|f''(t)|+\phi(t)\\
                               & = & c\phi(t)\, . 
\end{eqnarray*}
Therefore the second assertion has been established.
Now $\ds\phi(t)=-\frac{f'(t)}{t}$ so
$\ds \phi'(t)=\frac{-f''(t)t+f'(t)}{t^2}$ and $\ds t\phi'(t)=-f''(t)+\frac{f'(t)}{t}$ which implies 
\begin{eqnarray*}
|t\phi'(t)| & = & \left|\frac{f'(t)}{t}-f''(t)\right| \\
            & = & \left||f''(t)|-\left|\frac{f'(t)}{t}\right|\right|\\
            & \leq & \left|2\phi(t)-\phi(t)\right|=\phi(t)\: .
\end{eqnarray*} 
Thus the third assertion follows.

\end{proof}

\newtheorem{secondlemma}[mainsec]{Lemma}
\begin{secondlemma}
For a function $\ds\phi(t)$ satisfying the properties of Lemma ~\ref{lemone}, 
the following hold.
\begin{enumerate}
\item 
\begin{eqnarray*}
n\int_{-\delta}^{\delta}(1-kt^2\phi(t))^{n-1}|t|\phi(t)< C
\end{eqnarray*}
\item
\begin{eqnarray*}
n^2\int_{-\delta}^{\delta}(1-kt^2\phi(t))^{n-2}|t|^3\phi^2(t)< C
\end{eqnarray*}
\end{enumerate}
where $C$ is independent of $n$, and $\ds 0\leq 1-kt^2\phi(t)\leq 1$ on $(-\delta,\delta)$.
\end{secondlemma}

\begin{proof}
\begin{enumerate}
\item
\mbox{}
\begin{eqnarray*}
n\int_{-\delta}^{\delta} (1-kt^2\phi(t))^{n-1}|t|\phi(t)\dt & =
&   2n\int_{0}^{\delta}(1-kt^2\phi(t))^{n-1}t\phi(t)\dt\\
& = & \frac{2n}{-2k}\int_{0}^{\delta}(1-kt^2\phi(t))^{n-1}(-2kt\phi(t))\dt \\
& = & \frac{n}{-k}\int_{0}^{\delta}(1-kt^2\phi(t))^{n-1}(-2kt\phi(t)-kt^2\phi'(t))\dt\\
& + & \left(\frac{n}{k}\right)\int_{0}^{\delta}(1-kt^2\phi(t))^{n-1}(-kt^2\phi'(t))\dt\: .
\end{eqnarray*}
Therefore 
\begin{eqnarray*}
2n\int_0^{\delta}(1-kt^2\phi(t))^{n-1}t\phi(t)\dt & - & \frac{n}{k}\int_0^{\delta}(1-kt^2\phi(t))^{n-1}(-kt^2\phi'(t))\dt=\\
%& = & 2n\int_0^{\delta}(1-kt^2\phi(t))^{n-1}t\phi(t)\dt \\
%& + &  n \int_0^{\delta}(1-kt^2\phi(t))^{n-1}t^2\phi'(t)\dt \\
& =  & -\frac{n}{k}\int_0^{\delta}(1-kt^2\phi(t))^{n-1}(-2kt\phi(t)-kt^2\phi'(t))\dt \\
& = & \frac{n}{k}\int_{1-k\delta^2\phi(\delta)}^{1}u^{n-1}\,\mbox{d}u\leq C\: . 
\end{eqnarray*}
However, since $\ds |t\phi'(t)|\leq\phi(t)$ we have
\begin{eqnarray*}
2n\int_0^{\delta}(1-kt^2\phi(t))^{n-1}t\phi(t)\dt & + & n\int_0^{\delta}(1-kt^2\phi(t))^{n-1}t^2\phi'(t)\dt \\
& \geq & n\int_0^{\delta} (1-kt^2\phi(t))^{n-1}t\phi(t)\dt \\
& = & \frac{1}{2}n\int_{-\delta}^{\delta}(1-kt^2\phi(t))^{n-1}|t|\phi(t)\dt\: .
\end{eqnarray*}
The claim follows.

\item
\mbox{}
\begin{eqnarray*}
\lefteqn{n^2\int_{-\delta}^{\delta} (1-kt^2\phi(t))^{n-2}|t|^3\phi^2(t)\dt  =}\\
& =&   2n^2\int_{0}^{\delta}(1-kt^2\phi(t))^{n-2}t^3\phi^2(t)\dt\\
& = & 2n^2\int_{0}^{\delta}(1-kt^2\phi(t))^{n-2}t\phi(t)t^2\phi(t)\dt \\
& = & \frac{n^2}{-k^2}\int_{0}^{\delta}(1-kt^2\phi(t))^{n-2}(-2kt\phi(t)-kt^2\phi'(t))t^2\phi(t)\dt \\
& + & \left(\frac{n^2}{k}\right)\int_{0}^{\delta}(1-kt^2\phi(t))^{n-2}(-kt^2\phi'(t))t^2\phi(t)\dt
\end{eqnarray*}

Therefore by similar arguments as in the first part of this lemma, we have,
\begin{eqnarray*}
\lefteqn{\left|\frac{1}{2}n^2\int_{-\delta}^{\delta}(1-kt^2\phi(t))^{n-2}t^3\phi^2(t)\dt\right|  \leq}  \\
& \leq & -\frac{n^2}{k}\int_0^{\delta}(1-kt^2\phi(t))^{n-2}(-2kt\phi(t)-kt^2\phi'(t))t^2\phi(t)\dt\\
& = & \frac{n^2}{k}\int_{1-k\delta^2\phi(\delta)}^1u^{n-2}(1-u)\,\mbox{d}u\\
& = & \frac{n^2}{k^2}\int_{1-k\delta^2\phi(\delta)}^1 u^{n-2}-u^{n-1}\,\mbox{d}u
\end{eqnarray*}
\begin{eqnarray*}
& = & \frac{n^2}{k^2}\left(\left.\frac{u^{n-1}}{n-1}-\frac{u^n}{n}\right|_{1-k\delta^2\phi(\delta)}^1\right)\\
& = & \frac{n^2}{k^2}\left(\frac{1}{n-1}-\frac{1}{n}+\frac{(1-k\delta^2\phi(\delta))^n}{n}-\frac{(1-k\delta^2\phi(\delta))^{n-1}}{n-1}\right)\\
& \leq & \frac{n^2}{k^2}\left(\frac{1}{(n-1)n}+O\left(\frac{(1-k\delta^2\phi(\delta))^n}{n}\right)\right)\\
& \leq & C\frac{n^2}{k^2}\left(\frac{1}{n^2}+\frac{1}{n}O((1-k\delta^2)^n)\right)\\
& \leq & C\frac{n^2}{k^2}\left(\frac{1}{n^2}+\frac{1}{n}O(e^{-nk\delta^2})\right)\\
& \leq & C\frac{n^2}{k^2}\left(\frac{1}{n^2}+O\left(\frac{1}{n^3}\right)\right) \leq C
\end{eqnarray*}
\end{enumerate}
\end{proof}

\newtheorem{mainthm}[mainsec]{Theorem}
\begin{mainthm}
\label{mainthm}
Suppose that $\ds\theta(t)$ is twice differentiable in a neighborhood $\ds[-\delta,\delta]$ of $\:0$, except perhaps at $0$, and there exists 
a function $\ds \phi(t)$ satisfying the properties $\ds (1)-(3)$ of Lemma ~\ref{lemone} such that
\begin{enumerate}
\item
\mbox{}
$\ds|\theta(t)| \leq  1-kt^2\phi(t)$
\item
\mbox{}
$\ds\left|\frac{\thf(t)}{t}\right| \leq  c\phi(t)$
\item
\mbox{}
$\ds|\ths(t)| \leq  c\phi(t)$
\end{enumerate}
on $\ds [-\delta,\delta]$ for some $k\, , c>0$, then 
\begin{eqnarray*}
|\mu_n(x+y)-\mu_n(x)|\leq c\frac{|y|}{|x|^2} & \mbox{for} & n\geq |x|^{\delta/8}\,\mbox{and}\,
|y|\leq\frac{|x|}{2}
\end{eqnarray*}  
\end{mainthm}

\theoremstyle{definition}
\newtheorem{newremark}[mainsec]{Remark}
\begin{newremark}
If $\ds \theta''(t)$ is bounded on some $\ds [-\delta,\delta]$ then the function
$\ds \phi(t)=\sup_{t\in[-\delta,\delta]}\{|\theta''(t)|\}$ satisfies preporties $\ds (1)-(3)$
of Lemma~\ref{lemone} as well as $\ds (1)-(3)$ of Theorem~\ref{mainthm}. Thus Theorem~\ref{mainthm} applies to this case. Note that $(1)$
of Theorem~\ref{mainthm} follows from Theorem~\ref{petrov}.
\end{newremark}

\theoremstyle{plain}
\begin{proof}[Proof of Theorem~\ref{mainthm}]
By the inversion formula
\begin{eqnarray*}
\mu_n(x) & = & \ith \theta^n(t)e^{-2\pi ixt}\dt \\
& = & \int_{|t|\leq\delta}\theta^n(t)e^{-2\pi ixt}\dt +\int_{|t|\geq\delta}\theta^n(t)e^{-2\pi ixt}\dt\\
& = & I_{1,x}+I_{2,x}
\end{eqnarray*}

Now, if $\ds\theta^n(t) =\theta_n(t)$ we note that since $\ds\thf(t)=\int_{\epsilon}^t\ths(t)+\thf(\epsilon)$,
letting $\ds \epsilon\rightarrow 0$ we have $\ds\thf(t)=\int_0^t\ths(t)$ and $\ds\thf(t)$ is absolutely continuous
on $\ds [-\delta,\delta]$ hence

\begin{eqnarray*}
I_{1,x} & = & \frac{1}{2\pi i}\left(-\left.\frac{\theta_n(t)e^{-2\pi ixt}}{x}\right|_{-\delta}^{\delta}+\frac{1}{x}\int_{-\delta}^{\delta}\theta'_n(t)e^{-2\pi ixt}\dt\right)\\
& = & \frac{1}{2\pi i}\left(\left.-\frac{\theta_n(t)e^{-2\pi ixt}}{x}\right|_{-\delta}^{\delta}\right)+\frac{1}{4\pi ^2}\left(\left.\frac{\theta'_n(t)e^{-2\pi ixt}}{x^2}\right|_{-\delta}^{\delta}\right)\\
& - & \frac{1}{4\pi^2}\left(\frac{1}{x^2}\int_{-\delta}^{\delta}\theta''_n(t)e^{-2\pi ixt}\dt\right) \\
& = & Q_x -\frac{1}{4\pi^2x^2}\int_{-\delta}^{\delta}\theta''_n(t)e^{-2\pi ixt}\dt
\end{eqnarray*}
\begin{eqnarray*}
& = & Q_x-\frac{1}{4\pi x^2}\int_{-\delta}^{\delta}\theta''_n(t)(e^{-2\pi ixt}-1)\dt-\frac{1}{4\pi x^2}\int_{-\delta}^{\delta}\theta''_n(t)\dt\\
& = & Q_x -\frac{1}{4\pi^2x^2}\left(\theta'_n(\delta)-\theta'_n(-\delta)\right)-\frac{1}{4\pi^2x^2}\int_{-\delta}^{\delta}\theta''_n(t)(e^{-2\pi ixt}-1)\dt\\
& = & Q_x+P_x-\frac{1}{4\pi^2x^2}\int_{-\delta}^{\delta}\theta''_n(t)(e^{-2\pi ixt}-1)\dt\: .
\end{eqnarray*}

Therefore $$ \mu_n(x)=-\frac{1}{4\pi^2x^2}\int_{-\delta}^{\delta}\theta''_n(t)(e^{-2\pi ixt}-1)\dt +Q_x+P_x+I_{2,x}.$$

We have, for $\ds 2|y|\leq |x|$ by Lemma~\ref{caldlem},

\begin{eqnarray*}
\lefteqn{\left|\int_{-\delta}^{\delta}\theta''_n(t)\frac{e^{-2\pi i(x+y)t}-1}{(x+y)^2}\dt  -  \int_{-\delta}^{\delta}\theta''_n(t)\frac{e^{-2\pi ixt}-1}{x^2}\dt\right|} \\
& \leq & C\frac{|y|}{|x|^2}\int_{-\delta}^{\delta}|\theta''_n(t)||t|\dt\:\mbox{recalling that}\:\ths_n(t)=(\theta^n(t))''\:\mbox{and by}\:\ref{petrov}\:\mbox{and}\:\ref{resultcom}  \\
& \leq & C\frac{|y|}{|x|^2}\left[n(n-1)\int_{-\delta}^{\delta}|\theta^{n-2}(t)||\thf(t)|^2|t|\dt+n\int_{-\delta}^{\delta}|\theta^{n-1}(t)||\ths(t)||t|\dt\right]\\
& \leq & C\frac{|y|}{|x|^2}\left(n^2\int_{-\delta}^{\delta}(1-kt^2\phi(t))^{n-2}|t|^3\phi^2(t)\dt+n\int_{-\delta}^{\delta}(1-kt^2\phi(t))^{n-1}\phi(t)|t|\dt\right)\\
& \leq & C\frac{|y|}{|x|^2}. 
\end{eqnarray*}

Since $\ds |\theta^n(t)|\leq e^{-C nt^2}$, $\ds |\theta^n(-\delta)|$, $\ds|\theta^n(\delta)|\leq e^{-Cn\delta^2}$ and\\ 
$\ds |\thf(\delta)|$, $\ds |\thf(-\delta)|=|n\theta^{n-1}(\delta)\thf(\delta)|$, $\ds |n\theta^{n-1}(-\delta)\thf(-\delta)|\leq C n e^{-n\delta^2}$, we have
$\ds |I_{2,x}|,|P_x|,|Q_x|\leq C n e^{-n\delta^2}\leq \frac{C}{n^{16/\delta}}\leq \frac{c}{|x|^2}$
provided that $\ds n\geq |x|^{\delta/8}$. Since $\ds y\in\mathbb{Z}$ and $\ds |y|\leq \frac{|x|}{2}$ we have

$$|I_{2,x+y}-I_{2,x}|,|P_{x+y}-P_{x}|,|Q_{x+y}-Q_{x}|\leq \frac{c}{|x|^2}\leq C\frac{|y|}{|x|^2}.$$

The theorem follows.
\end{proof}

\noindent Combining \ref{corm}, \ref{resultcom}, \ref{lemone} and \ref{mainthm} we obtain.
\newtheorem{finalthm}[mainsec]{Theorem}
\begin{finalthm}
Let $\ds\mu$ be a probability with $\ds m_1(\mu)<\infty$, bounded angular ratio and $\ds\sum_{|k|\leq n}k^2\mu(k)=O(n^{1-\delta})$ for some $\ds 0<\delta\leq 1$. 
Suppose $\ds\theta(t)=f(t)+ig(t)$ is the Fourier transform of $\ds\mu$ and $\ds\ths(t)$ exists at all points except possibly $0$
in a neighborhood $\ds [-\delta,\delta]$. Then if $\ds f''(t)=p(t)+O(1)<0$ where $p(t)$ is non-decreasing on $S$, there exists $\ds 0<\alpha\leq 1$ such that 
$$|\mu^n(x+y)-\mu^n(x)|\leq C\frac{|y|^{\alpha}}{|x|^{1+\alpha}}\:,\forall\:|y|\leq\frac{|x|}{2}.$$
\end{finalthm}

%\begin{proof}
%We need only consider the case where $\ds m_2(\mu)=\infty$. In this case $\ds \rho(t)\rightarrow -\infty$ as
%$\ds |t|\rightarrow 0$. We first show that $\ds\left|\frac{f''(t/2)}{f''(t)}\right|\geq c>0$.

%We have
%\begin{eqnarray*}
%\left|\frac{f''(t/2)}{f''(t)}\right| & = & \frac{|p(t/2)+O(1)(t/2)|}{|p(t)+O(1)(t)|}\\
%& \geq & \frac{|p(t/2)|-|O(1)(t/2)|}{|p(t)|+|O(1)(t)|}\\
%& \geq & \frac{|p(t/2)|}{|p(t)|+|p(t)|}-O(|1/p(t)|)\geq 1/4 
%\end{eqnarray*}

%Therefore Proposition~\ref{resultcom} gives $\ds |\theta(t)|\leq 1-c\phi(t) t^2$ where $\phi(t)$
%satisfies the properties of Lemma~\ref{lemone}. Therefore Theorem~\ref{mainthm} gives
%$$|\mu_n(x+y)-\mu_n(x)|\leq c\frac{|y|}{|x|^2}\, ,\: |y|\leq |x|/2,\:\mbox{for}\: n\geq |x|^{\delta/8}.$$
%The condition $\ds \sum_{|k|\leq n}|k|^2\mu(k)=O(n^{1-\delta})$ gives, by Corollary ~\ref{corm}, the
%inequality when $\ds n\leq |x|^{\delta/8}$.
%\end{proof}

%\newtheorem{bjr}[result1]{Theorem}
%\begin{bjr}[Bellow-Jones-Rosenblatt]
%If $\ds\mu$ is symmetric and strictly aperiodic, then for all $\ds f\in\mbox{L}^p(X)$, $\ds 1<p<\infty$,
%there exists a unique $\ds\tau-$invariant function $\ds f^{\ast}\in\mbox{L}^p(X)$ such that
%$\ds\lim_{n\rightarrow \infty}\mu^nf(x)=f^{\ast}(x)$ for a.e $x$. If $\ds\tau$ is ergodic and $\ds m(X)<\infty$
%then $\ds f^{\ast}(x)=\int f\,\mbox{d}m$ for a.e $x$. 
%\end{bjr}

\bibliographystyle{plain}
\bibliography{ChrispaperFRbib}

\end{document}